\pdfoutput=1
\documentclass[11pt,a4paper]{article}
\usepackage[utf8]{inputenc} % Allows to use UTF8 encoded symbols directly (e.g.  ä ö ü ß é)
\usepackage[english]{babel} % English Language settings

\usepackage{lmodern}
\usepackage[T1]{fontenc}
\usepackage{amsmath,mathtools,amsfonts,amssymb} % AMS Math fonts and symbols
\usepackage{slashed} % Slashed symbols e.g. to typeset Dirac operators
\usepackage{mathrsfs}
\usepackage[babel=true]{microtype}
\usepackage[autostyle=true]{csquotes} % Correctly typeset via quotes via \enquote

\usepackage{comment}

\usepackage[a4paper,marginparwidth=2cm]{geometry}

\usepackage{amsthm}
\usepackage{thmtools}
\usepackage[pdfusetitle]{hyperref}
\usepackage{titlesec}
\titleformat*{\section}{\large\scshape}
\titleformat{\subsection}[runin]{\normalsize\bfseries}{\thesubsection}{5pt}{}[.]

\usepackage[citestyle=numeric-comp,
            bibstyle=numeric-comp,
            backend=biber,
            url=false,
            doi=true,
            isbn=false,
            giveninits=true,
            maxbibnames=100,
            sorting=nyt]{biblatex}
\AtEveryBibitem{\clearfield{month}}
\AtEveryBibitem{\clearfield{day}}
\AtEveryBibitem{%
  \ifentrytype{thesis}
    {}
    {
      \ifentrytype{online}{}
      {
      \clearfield{url}%
      \clearfield{urldate}%
      }
    }%
}
            
\usepackage{tikz}
\usetikzlibrary{cd} % Commutative diagrams

\usepackage[shortlabels]{enumitem}
\newlist{myenumi}{enumerate}{1}
\setlist[myenumi,1]{label=\upshape(\roman*)}
\newlist{myenuma}{enumerate}{1}
\setlist[myenuma,1]{label=\upshape(\alph*)}
\usepackage{color}
\usepackage[disable]{todonotes}
\declaretheorem[name=Theorem, numberwithin=section]{thm}
\declaretheorem[name=Theorem, numbered=no]{thm*}
\declaretheorem[name=Lemma,numberlike=thm]{lem}
\declaretheorem[name=Lemma,numbered=no]{lem*}

\declaretheorem[name=Proposition,numberlike=thm]{prop}
\declaretheorem[name=Definition,numberlike=thm, style=definition]{defi}

\declaretheorem[name=Conjecture,numberlike=thm, style=remark]{conj}

\declaretheorem[name=Example, numberlike=thm, style=remark]{ex}

\declaretheorem[name=Theorem]{thmx}
\declaretheorem[name=Corollary, numberlike=thmx]{corx}

\numberwithin{equation}{section}
\allowdisplaybreaks[1]
\usepackage[noabbrev]{cleveref} %Important to load after numberwithin!
\crefname{figure}{Figure}{Figures}
\crefname{table}{Table}{Tables}
\crefname{thm}{Theorem}{Theorems}
\crefname{thmx}{Theorem}{Theorems}
\crefname{lem}{Lemma}{Lemmas}
\crefname{defi}{Definition}{Definitions}
\crefname{setup}{Setup}{Setups}
\crefname{conj}{Conjecture}{Conjectures}
\crefname{quest}{Question}{Questions}
\crefname{cor}{Corollary}{Corollaries}
\crefname{corx}{Corollary}{Corollaries}
\crefname{prop}{Proposition}{Propositions}
\crefname{ex}{Example}{Examples}
\crefname{rem}{Remark}{Remarks}
\crefname{section}{Section}{Sections}
\crefname{subsection}{Subsection}{Subsections}
\crefname{chapter}{Chapter}{Chapters}
\crefname{appendix}{Appendix}{Appendices}
\crefdefaultlabelformat{#2\textup{#1}#3}
\creflabelformat{enumi}{(#2#1#3)}
\usepackage{xparse}
\usepackage{ourmacros}

\usepackage{authblk}

\title{The positive mass theorem and distance estimates\\in the spin setting}
\author{Simone Cecchini\thanks{Funded by the Deutsche Forschungsgemeinschaft (DFG, German Research Foundation) through the Priority Programme ``Geometry at Infinity'' (SPP 2026, CE~393/1-1).}}
\affil{Mathematical Institute\\University of Göttingen, Germany\\\vspace{0.2cm}
email:~\href{mailto:simone.cecchini@mathematik.uni-goettingen.de}{simone.cecchini@mathematik.uni-goettingen.de}\\ url:~\href{https://simonececchini.org}{simonececchini.org}}
\author{Rudolf Zeidler\thanks{Funded by the Deutsche Forschungsgemeinschaft (DFG, German Research Foundation) – Project-ID 427320536 – SFB 1442, as well as under Germany’s Excellence Strategy EXC 2044  390685587, Mathematics Münster: Dynamics–Geometry–Structure, and through the Priority Programme ``Geometry at Infinity'' (SPP 2026, ZE~1123/2-2).}}
\affil{Mathematical Institute\\University of Münster, Germany\\\vspace{0.2cm}
email:~\href{mailto:math@rzeidler.eu}{math@rzeidler.eu}\\ url:~\href{https://www.rzeidler.eu}{www.rzeidler.eu}}
\date{}
\hypersetup{
  colorlinks=true,
  allcolors=blue,
}
\begin{document}

\maketitle

\begin{abstract} % Do not use any macros or advanced TeX in the abstract to be usable for the arXiv listing
Let $\mathcal{E}$ be an asymptotically Euclidean end in an otherwise arbitrary complete and connected Riemannian spin manifold $(M,g)$.
We show that if $\mathcal{E}$ has negative ADM-mass, then there exists a constant $R > 0$, depending only on $\mathcal{E}$, such that $M$ must become incomplete or have a point of negative scalar curvature in the $R$-neighborhood around $\mathcal{E}$ in $M$.
This gives a quantitative answer to Schoen and Yau's question on the positive mass theorem with arbitrary ends for spin manifolds. Similar results have recently been obtained by Lesourd, Unger and Yau without the spin condition in dimensions $\leq 7$ assuming Schwarzschild asymptotics on the end $\mathcal{E}$.
We also derive explicit quantitative distance estimates in case the scalar curvature is uniformly positive in some region of the chosen end $\mathcal{E}$.
Here we obtain refined constants reminiscent of Gromov's metric inequalities with scalar curvature.
\end{abstract}
\newpage
% \tableofcontents
% \listoftodos

\section{Introduction and main results}
The positive mass theorem of Schoen and Yau is a central result and an exceedingly useful tool in the study of the geometry of scalar curvature, mathematical relativity and geometric analysis.
The protagonist in the Riemannian version of the positive mass theorem is the \emph{ADM-mass}, originally attributed to \citeauthor{ADM}~\cite{ADM}, a physically interesting geometric invariant \(\mass(\mathcal{E},g) \in \R\) associated to each end \(\mathcal{E} \subseteq M\) in an \emph{asymptotically Euclidean} manifold \((M,g)\).
Note we use the terminology \enquote{asymptotically Euclidean} instead of the commonplace \enquote{asymptotically flat} because we feel the former is slightly less ambiguous, but these terms should be understood as synonyms for the purposes of this paper; for details we refer to \cref{SS:AsymptoticallyFlat}.
\begin{thm}[Riemannian Positive Mass Theorem]\label{thm:PMT}
Let $(M,g)$ be a complete asymptotically Euclidean manifold \parensup{without boundary} of dimension \(n \geq 3\) which has non-negative scalar curvature.
Then the ADM-mass of each end of $(M,g)$ is non-negative.
Furthermore, if \(M\) has an end of zero mass, then $(M,g)$ is isometric to Euclidean space.
\end{thm}
\Cref{thm:PMT} was first proved by Schoen and Yau~\cite{SY79I,SY79II,Schoen-Yau81:General_Asymptotics} for $n\leq 7$ using minimal hypersurfaces, whereas the higher-dimensional cases have only recently been treated in preprints of \citeauthor{schoen2017positive}~\cite{schoen2017positive} and Lohkamp~\cite{LoI,LoII}.
A different method based on spinors was found by Witten~\cite{Witten:positive_mass} which allows to prove \cref{thm:PMT} in all dimensions provided that \(M\) is spin; see also~\cite{Parker-Taubes:Wittens_Proof,Bartnik:MassAsymptoticallyFlat}.

A curious feature of \cref{thm:PMT} is that it consists of separate statements for each asymptotically Euclidean end of \(M\).
Thus one is lead to ask if the positive mass theorem can, in a certain sense, be localized to a single end without imposing the strong topological and metric restrictions of being asymptotically Euclidean on the remaining part of the manifold:
% \noindent In this paper, we establish a version of the positive mass theorem localized to a single asymptotically Euclidean end.
% We start with the following conjecture.
% 
\begin{conj}[{Positive Mass Conjecture with Arbitrary Ends~\cite{lesourd2020positive,SY88}}]\label{conj:PMTAE}
Let $(M,g)$ be an arbitrary complete Riemannian manifold of dimension \(n \geq 3\) which has non-negative scalar curvature.
Let $\mathcal E \subseteq M$ be a single asymptotically Euclidean end in $M$.
Then $\mass(\mathcal E,g)\geq0$.
\end{conj}
Beyond being a natural geometric question, the interest in this conjecture goes back to seminal work of \citeauthor{SY88}~\cite{SY88} in the 1980s on conformally flat manifolds, where it is shown that \cref{conj:PMTAE} implies a Liouville conjecture for locally conformally flat manifolds of non-negative scalar curvature; see~\cite[\S4]{SY88} and the discussion in~\cite[\S1]{lesourd2020positive} for details.
However, the last remaining cases of the Liouville theorem have recently been proved through a combination of preprints by \citeauthor{chodosh_li}~\cite{chodosh_li} and \citeauthor{lesourd2020positive}~\cite{lesourd2020positive}, but without addressing \cref{conj:PMTAE}.
Instead, in a separate more recent preprint, \citeauthor{lesourd2021positive}~\cite{lesourd2021positive} proved \cref{conj:PMTAE} for \(n \leq 7\) assuming that the chosen end \(\mathcal{E}\) satisfies a more specific fall-off condition, namely that it is \emph{asymptotic to Schwarzschild} (compare \cref{ex:Schwarzschild}).
In the classical case of \cref{thm:PMT}, more general fall-off conditions can be reduced to Schwarzschild asymptotics via a density theorem, but at the moment this does not appear to be readily available in the setting of arbitrary ends as pointed out in \cite{lesourd2021positive}.

In the spin setting, \cref{conj:PMTAE} has been addressed by \citeauthor{BC05} who showed that Witten's method can be applied directly on non-compact manifolds with arbitrary ends via the methods developed in~\cite{BC05}; more precisely see~\cite[Theorem~11.2]{BC03}.
Moreover, their methods also yield the related rigidity statement for the case of zero mass, that is, $\mass(\mathcal E,g) = 0$ implies that \((M,g)\) is flat and hence isometric to Euclidean space.

%Thus, so far, \cref{conj:PMTAE} still remains open under the general fall-off conditions now commonly used in the context of the positive mass theorem as well as in higher dimensions.  
%Furthermore, the second part of \cref{thm:PMT} consists of a rigidity statement for the case of zero mass.
%It appears natural to ask if the analogous assertion holds in the setting of arbitrary ends.
%More precisely, in the situation of \cref{conj:PMTAE}, one can ask whether $\mass(\mathcal E,g) = 0$ implies that \((M,g)\) is flat and hence isometric to Euclidean space.

The main goal of the present paper is to study \cref{conj:PMTAE} in the spin setting from a quantitative point of view via an augmentation of Witten's method.
Indeed, our first main theorem can be viewed as a quantitative refinement of the statement of \cref{conj:PMTAE} which makes the idea of \enquote{localizing} the positive mass theorem precise.
\begin{thmx}\label{thm:A}
Let \((\mathcal{E},g)\) be an \(n\)-dimensional asymptotically Euclidean end such that $\mass(\mathcal E,g)< 0$.
Then there exists a constant \(R = R(\mathcal{E},g)\) such that the following holds: If \((M,g)\) is an \(n\)-dimensional Riemannian manifold without boundary that contains \((\mathcal{E}, g)\) as an open subset and \(\Nbh = \Nbh_R(\mathcal{E}) \subseteq M\) denotes the open neighborhood of radius \(R\) around \(\mathcal{E}\) in \(M\), then at least one of the following conditions must be violated:
\begin{myenuma}
    \item \(\overline{\Nbh}\) is \parensup{metrically} complete, \label{item:complete}
    \item \(\inf_{x \in \Nbh} \scal_g(x) \geq 0\), \label{item:geq_scal}
    \item \(\Nbh\) is spin. \label{item:spin}
\end{myenuma}
\end{thmx}
Intuitively this means that, if the mass of an asymptotically Euclidean end \(\mathcal{E}\) is negative, then the hypotheses of (the spin proof of) the positive mass theorem must be violated in the \(R\)-neighborhood around \(\mathcal{E}\).
Here \(R = R(\mathcal{E},g)\) may be large but the crucial feature is that it is a constant which only depends on \(\mathcal{E}\) and not on the entire ambient manifold \(M\).
Of course, condition \labelcref{item:spin} should be viewed as conjecturally redundant, and can be dropped from the statement of \cref{thm:A} if \(M\) is assumed to be spin.
Nevertheless it is intriguing to formulate the result in this way because it shows that an augmentation of Witten's method can address certain non-spin manifolds provided that the reason for being non-spin is sufficiently far away from the chosen end.
\Cref{thm:A} implies \cref{conj:PMTAE} for spin manifolds as a formal consequence.
In addition, it turns out that our proof of \cref{thm:A} allows us to also address zero mass rigidity.
Combining these statements, we therbey obtain an alternative proof of (the Riemannian version of) the result due to \citeauthor{BC03}~\cite[Theorem~11.2]{BC03} on \cref{conj:PMTAE} in the spin setting:
\begin{thmx}\label{thm:B}
Let $(M,g)$ be a complete connected $n$-dimensional Riemannian spin manifold without boundary such that $\scal_g\geq 0$ and let $\mathcal E \subset M$ be an asymptotically Euclidean end.
Then $\mass(\mathcal E,g)\geq 0$.
Moreover, if \(\mass(\mathcal E,g) = 0\), then \((M,g)\) is flat and thus isometric to Euclidean space. 
\end{thmx}

The value of the number \(R(\mathcal{E},g) > 0\) appearing in \cref{thm:A} could in principle be traced through our proof but it crucially depends on the constant in a weighted Poincaré inequality on the chosen asymptotically Euclidean end and thus it appears difficult to make explicit.
Note that \citeauthor{Lesourd-Unger-Yau:Positive_mass_ends} have also established a quantitative theorem in their approach to \cref{conj:PMTAE} which involves more explicit estimates; see~\cite[Theorem~1.6]{lesourd2021positive}.
However, unlike our \cref{thm:A}, their quantitative theorem relies on a largeness assumption on the scalar curvature which forces it to be strictly positive in a certain region.
Therefore in \cite{lesourd2021positive} an intermediary step involving a conformal perturbation of the metric is necessary in order to attack \cref{conj:PMTAE}.
Since in the situation of \cref{conj:PMTAE} the part outside of the chosen end is an arbitrary complete manifold, this is a subtle point and here the stronger Schwarzschild asymptotics assumed in \cite{lesourd2021positive} play an important role.
In comparison, our approach completely circumvents this issue by leaving the metric in place.
Instead, we perform a careful perturbation of the Dirac operator used in Witten's approach to localize the problem in a neighborhood of the chosen end \(\mathcal{E}\).
In this way, we are able to prove our results under asymptotically Euclidean asymptotics and establish rigidity.

Nevertheless, if we do assume that the scalar curvature satisfies a suitable largeness assumption, then our method also allows to prove explicit estimates inspired by the quantitative theorem of \citeauthor{Lesourd-Unger-Yau:Positive_mass_ends}~\cite[Theorem~1.6]{lesourd2021positive}.
Another feature is that under the hood we work with manifolds which have a non-empty interior boundary and we impose conditions on the boundary mean curvature (which, crucially, may be negative).
This leads to the following result which relates positivity of the ADM-mass to a quantitative relationship between a positive lower bound on the scalar curvature in a certain region, the width of this region, its distance to the boundary, and a negative lower bound on mean curvature of the boundary.
\begin{thmx}\label{thm:C}
Let $(X,g)$ be an $n$-dimensional complete asymptotically Euclidean spin manifold of nonnegative scalar curvature with compact boundary.
Let $X_0 \subseteq X_1 \subseteq X$ be codimension zero submanifolds with boundary such that $X_0$ contains all asymptotically Euclidean ends of $X$.
Moreover, we assume that \(\scal_g \geq \kappa n(n-1)\) on \(X_1 \setminus X_0\) for some \(\kappa >0\).
We let \(d = \dist_g(\partial X_0, \partial X_1)\) and \(l = \dist_g(\partial X_1, \partial X)\) and define
\[
    \Psi(d,l) \coloneqq \begin{cases}  
        \frac{2}{n} \frac{\lambda(d)}{1-l \lambda(d)} & \text{if \(d < \frac{\pi}{\sqrt{\kappa} n}\) and \(l < \frac{1}{\lambda(d)}\),} \\
        \infty & \text{otherwise,}
     \end{cases} 
     \qquad \text{where \(\lambda(d) \coloneqq \frac{\sqrt{\kappa} n}{2} \tan\left(\frac{ \sqrt{\kappa} n d}{2}\right)\)}.
\]
In this situation, if the mean curvature of \(\partial X\) satisfies
\[
    \mean_g > - \Psi(d,l) \quad \text{on \(\partial X\)},
\]
then the ADM-mass of each end of $X$ is strictly positive.
\end{thmx}

While we have formulated this result for globally asymptotically Euclidean manifolds, this again yields as a localized statement by taking \(X\) to be the manifold obtained from cutting off everything outside a sufficiently large neighborhood of a single chosen end.
Indeed, since \(\Psi(d,l)\) tends to \(+ \infty\) as either \(d\) or \(l\) reaches an explicit finite threshold, the restriction on the mean curvature of thereby created boundary components eventually becomes redundant.
In particular, we obtain the following corollary which may be viewed as a slightly sharper variant of \cite[Theorem~1.6]{lesourd2021positive} in the spin setting.

\begin{corx}\label{cor:UpperBound}
    Let $(X,g)$ be an $n$-dimensional asymptotically Euclidean spin manifold of nonnegative scalar curvature with compact boundary.
    Let $X_0 \subseteq X_1 \subseteq X$ be codimension zero submanifolds with boundary such that $X_0$ contains all asymptotically Euclidean ends of $X$.
    Moreover, we assume that \(\scal_g \geq \kappa n(n-1)\) on \(X_1 \setminus X_0\) for some \(\kappa >0\).
    We let \(d = \dist_g(\partial X_0, 
    \partial X_1)\) and \(l = \dist_g(\partial X_1, \partial X)\).
    In this situation, if 
    \[
        \text{either \(d \geq \frac{\pi}{\sqrt{\kappa} n}\)} \qquad \text{or \(l \geq \frac{1}{\lambda(d)}\), where \(\lambda(d) \coloneqq \frac{\sqrt{\kappa} n}{2} \tan\left(\frac{ \sqrt{\kappa} n d}{2}\right)\),}
    \]
     then the ADM-mass of each end of $X$ is strictly positive.
\end{corx}

We note that these estimates are completely analogous to Gromov's metric inequalities with scalar curvature~\cite{Gromov:MetricInequalitiesScalar} and even feature the same constants.
In particular, if $X_1=X$, \cref{cor:UpperBound} states that if $\scal_g\geq n(n-1)$ in $X\setminus X_0$ and $\dist_g(X_0,\partial X) \geq \frac{\pi}{n}$, then the mass of each end must be positive.
This can be interpreted as a \enquote{long neck principle}, originally proposed by Gromov for certain compact manifolds with boundary (compare~\cite[87]{gromovFourLecturesScalar2019v3} and \cite{Ce20}), in the context of the positive mass theorem.
Indeed, on a technical level, the present paper combines Witten's proof of the positive mass theorem with the technique systematically developed by the authors in \cite{Cecchini-Zeidler:ScalarMean} of using Callias operators on spin manifolds to obtain scalar- and mean curvature comparison results related to Gromov's metric inequalities programme.
See also~\cite{Zeidler:band-width-estimates,Guo-Xie-Yu:Quantitative,xie2021quantitative,wang2021proof,raede2021scalar,chodosh2021classifying,perspectives-in-psc:generalized-callias} for other related work in this area.

The paper is organized as follows.
In \cref{sec:Callias}, we develop the theory of Callias operators on asymptotically Euclidean manifolds with compact boundary.
\cref{sec:Long_Neck} is devoted to proving \cref{thm:C}.
Finally, in \cref{sec:PMTAE} we prove \cref{thm:A,thm:B}.

\subsection*{Acknowledgements}
The authors acknowledge the Oberwolfach Research Institute for Mathematics for its hospitality during the 2021 workshop \enquote{Analysis, Geometry and Topology of Positive Scalar Curvature Metrics}, where preliminary results from this paper were presented and part of the work was completed.
We thank Romain Gicquaud and Klaus Kröncke for helpful discussions and comments.
We are also grateful to Piotr Chru\'{s}ciel for pointing us to his work with Bartnik~\cite{BC03}.
\section{Callias operators in asymptotically Euclidean manifolds}\label{sec:Callias}

In this section, we study Callias operators on complete asymptotically Euclidean spin manifolds.
We also briefly discuss the necessary concepts of asymptotically Euclidean manifolds and mass.

\subsection{Asymptotically Euclidean manifolds and mass}\label{SS:AsymptoticallyFlat}
In this subsection, we recall the notions of an asymptotically Euclidean end its ADM-mass.
\begin{defi}\label{defi:AE_end}
    Let \((M,g)\) be a smooth \(n\)-dimensional Riemannian manifold.
    We say an open subset \(\mathcal{E} \subseteq M\) is an \emph{asymptotically Euclidean (AE) end} of order \(\tau > (n-2)/2\) if \(\scal_g\) belongs to \(\Lp^1\) on \(\mathcal{E}\) and there exists a diffeomorphism \(\Phi \colon \mathcal{E} \xrightarrow{\cong} \R^n \setminus \Disk_{d}(0)\) for some \(d > 0\) such that $\Phi^\ast g = \sum_{i,j=1}^n g_{ij}\ \D x^i \otimes \D x^j$ satisfies
\[
    g_{ij} - \delta_{ij}\in \Ct^2_{-\tau}(\R^{n} \setminus \Disk_d(0)),
\]
for all \(1 \leq i,j \leq n\), where \(\Ct^2_{-\tau}\) denotes the weighted \(\Ct^2\)-space of order \(-\tau\), that is, the space of \(\Ct^2\)-functions \(f\) such that the function \(|x|^{\tau+|\alpha|}| \partial^\alpha f|\) is bounded for each multi-index \(0 \leq |\alpha| \leq 2\).
\end{defi}

Once we fix such a diffeomorphism \(\Phi \colon \mathcal{E} \xrightarrow{\cong} \R^n \setminus \Disk_{d}(0)\) for an asymptotically Euclidean end $\mathcal{E}$, we denote the corresponding coordinates by $x = (x^1, \dotsc, x^n)$ and the radial coordinate by $\rho=|x|$.
Following the convention in~\cite[\S3.1.4]{Lee2019-GeometricRelativity}, we define the \emph{ADM-mass} of an AE end $\mathcal{E}$ as the limit
\begin{equation}\label{eq:ADM}
    \mass(\mathcal{E},g)\coloneqq\frac{1}{2(n-1)\omega_{n-1}}\lim_{r\to\infty}\int_{\Sphere^{n-1}_r}\sum_{i,j=1}^n\frac{x^j}{\rho}(\partial_ig_{ij}-\partial_jg_{ii})\,\dSbar,
\end{equation}
where $\omega_{n-1}$ is the volume of the unit \((n-1)\)-sphere, $\Sphere^{n-1}_r \subseteq \mathcal{E}$ is the sphere of radius $r$ with respect to the chosen asymptotically Euclidean coordinates \((x^1, \dotsc, x^n)\), and \(\dSbar\) denotes the volume element on \(\Sphere^{n-1}_r\) with respect to the Euclidean background metric.
Since our \cref{defi:AE_end} already includes the usual mass decay conditions, the quantity \(\mass(\mathcal{E},g) \in \R\) is well-defined, that is, the limit in \labelcref{eq:ADM} exists and is independent of the chosen asymptotically Euclidean coordinate chart \(\Phi\) by foundational results of \citeauthor{Bartnik:MassAsymptoticallyFlat}~\cite{Bartnik:MassAsymptoticallyFlat} and \citeauthor{Chrusciel:Spatial_infinity}~\cite{Chrusciel:Spatial_infinity}.

\begin{ex}\label{ex:Schwarzschild}
    The \emph{Schwarzschild metric} of mass \(m \in \R\), defined in isotropic coordinates by
    \begin{equation}
        (g_m)_{ij} \coloneqq \left(1 +  \frac{m}{2 |x|^{n-2}}\right)^{\frac{4}{n-2}} \delta_{ij} \quad \text{on \(\R^n \setminus \Disk_{d}(0)\)}, \label{eq:schwarzschild}
    \end{equation}
    where \(d\) must be chosen sufficiently large if \(m < 0\), yields a scalar-flat AE end \((\mathcal{E}, g_m)\) such that \(\mass(\mathcal{E}, g_m) = m\).
    More generally, an end \((\mathcal{E}, g)\) is called \emph{asymptotic to Schwarzschild} of mass \(m \in \R\) if in some asymptotically Euclidean coordinate chart the metric satisfies \(g_{ij} - (g_m)_{ij} \in \Ct^2_{1-n}\), where \(g_m\) is as in \labelcref{eq:schwarzschild}.
     An end which is asymptotic to Schwarzschild of mass \(m\) is always AE of order \(\tau = (n-2)\) and has ADM-mass \(m\).
\end{ex}

\begin{defi}\label{defi:asymptotically_flat}
A Riemannian manifold $(X,g)$ with compact boundary is said to be \emph{asymptotically Euclidean} (AE) if there exists a bounded subset $K\subset X$ whose complement $X \setminus K$ is a non-empty disjoint union of finitely many asymptotically Euclidean ends \(\mathcal{E}_1, \dotsc, \mathcal{E}_N \subseteq X\).
\end{defi}

Note that our definition of an AE manifold explicitly allows for \(X\) to have a compact interior boundary \(\partial X\) (which may be empty).

\subsection{Weighted Sobolev spaces on AE manifolds}
Let $(X,g)$ be an asymptotically Euclidean manifold with compact boundary.
%We denote by $\dV$ and $\dS$ the measures induced by $g$ respectively on $X$ and $\partial X$.
%Let $(E,\nabla)$ be a Hermitian vector bundle with metric connection on \(X\).
%The space of smooth sections will be denoted by \(\Ct^\infty(M,E)\) and the subspace of compactly supported smooth sections by $\Ct^\infty_\cpt(M,E)$.
%We denote %fiberwise inner products by $\langle\blank,\blank\rangle$ and 
%fiberwise norms by $|\blank|$.
%For sections $u$, $v\in \Ct^\infty_\cc(X,E)$, we let
%\[
%    \bigl(u,v\bigr)\coloneqq\int_X\langle u,v\rangle\dV
%\]
%denote the $\Lp^2$-inner product.
We will use weighted Sobolev spaces on \(X\) with coefficients in vector bundles.
To this end, fix a positive smooth function \(\rho \colon X \to (0,\infty)\) such that \(\rho = |x|\) outside a disk in each AE end with respect to some asymptotically Euclidean coordinate systems.
Moreover, we shall assume that \(\rho\) remains uniformly bounded away from \(0\) and \(\infty\) outside of the AE ends (this is automatic if \((X,g)\) is complete).
Let $(E,\nabla)$ be a Hermitian vector bundle with metric connection on \(X\).
For $p\geq 1$ and $\delta\in\R$, we define the \emph{weighted Lebesgue space} as the space of $u\in \Lp^p_\loc(X,E)$ such that the weighted norm
\begin{equation}
    \|u\|_{\Lp^p_\delta(X,E)}\coloneqq \begin{cases}
    \left(\int_X|u|^p\rho^{-\delta p-n} \dV\right)^{1/p} & p < \infty, \\
    \operatorname{ess\ sup}_{x \in X} |u(x)| \rho(x)^{-\delta} & p = \infty
    \end{cases}
\end{equation}
is finite.
For any $k\in\Z_{\geq 0}$, we define the \emph{weighted Sobolev space} $\SobolevW^{k,p}_\delta(X,E)$ as the space of sections $u\in \SobolevW^{k,p}_\loc(X,E)$ such that the weighted Sobolev norm
\begin{equation}
    \|u\|_{\SobolevW^{k,p}_\delta(X,E)}\coloneqq\sum_{i=0}^k\|\nabla^iu\|_{\Lp^p_{\delta-i}(X,E)}
\end{equation}
is finite.
In the case \(p = 2\), we use the usual notation \(\SobolevH^k_\delta \coloneqq \SobolevW^{2,k}_\delta\).

Note that, while these  norms of course depend on the chosen weight function \(\rho\) and thus implicitly on the chosen asymptotically Euclidean coordinates, different choices will lead to equivalent norms.
Insofar as we make any statements involving values of these norms, we shall assume that \(\rho\) has been fixed in advance.

\subsection{Callias operators}\label{sec:witten_relative_dirac_bundle}
In this subsection, we introduce the spinor Dirac operator augmented with a suitable potential which is the main new ingredient in our proofs.
We follow the formal setup from \cite[\S2]{Cecchini-Zeidler:ScalarMean}.

Let $(X,g)$ be a complete asymptotically Euclidean spin manifold with (possibly empty) boundary.
Let $\ReducedSpinBdl\to X$ be the complex spinor bundle on $X$.
Then \(S \coloneqq\ReducedSpinBdl\oplus\ReducedSpinBdl \) becomes a \(\Z/2\)-graded Dirac bundle if we endow it with the direct sum connection \(\nabla = \nabla_{\ReducedSpinBdl} \oplus \nabla_{\ReducedSpinBdl}\), and the Clifford multiplication
\(\clm(\xi)(u_1 \oplus u_2) = (\clm_{\ReducedSpinBdl}(\xi)u_2, \clm_{\ReducedSpinBdl}(\xi)u_1) \), 
where \(\clm_{\ReducedSpinBdl}\) and \(\nabla_{\ReducedSpinBdl}\) are respectively the Clifford multiplication and connection on \(\ReducedSpinBdl\).
Together with the involution
\begin{equation}\label{E:GLepsilon}
    \RelDiracInv \coloneqq \begin{pmatrix} 0 & -\iu \\ \iu & 0 \end{pmatrix} 
\end{equation}
it becomes a relative Dirac bundle in the sense of~\cite[\S2]{Cecchini-Zeidler:ScalarMean}.
The Dirac operator on $S$ is given by
\[
    \Dirac = \begin{pmatrix} 0 & \ReducedSpinDirac \\ \ReducedSpinDirac & 0 \end{pmatrix},
\]
where $\ReducedSpinDirac\colon\Ct^\infty(X,\ReducedSpinBdl)\to \Ct^\infty(X,\ReducedSpinBdl)$ is the spinor Dirac operator on $(X,g)$.
For a function \(\psi \in \Ct^\infty_\cc(X,\R)\), we consider the associated Callias operator
\begin{equation}
    \Callias_\psi\coloneqq\Dirac+\psi\sigma.
\end{equation}
Then direct calculation and Schrödinger--Lichnerowicz formula shows
\[
    \Callias_\psi^2=\Dirac^2+\clm(\D\psi)\sigma+\psi^2=\nabla^\ast\nabla+\frac{\scal_g}{4}+\clm(\D\psi)\sigma+\psi^2.
\]
Let \(\nu\) be the inward-pointing unit normal vector field along \(\partial X\).
Together with \(\RelDiracInv\) this defines the \emph{chirality operator} 
\[
    \chi \coloneqq \clm(\nu^\flat) \RelDiracInv \colon S|_{\partial X} \to S|_{\partial X}.
\]
In the following, we use \(+1\)-eigenbundle of \(\chi\) to define boundary conditions–note that this yields an elliptic boundary condition; see e.g.~\cite[Example~7.26]{Baer-Ballmann:Boundary-value-problems-first-order}.
Similar boundary conditions have been applied previously in the context of the positive mass theorem; see e.g.~\cite{GHHM:PMT_BlackHoles,Herzlich:BlackHoles}.
We use the notation \(\Ct^\infty(X,S;\chi)\) to denote the space of all smooth sections \(u\) of \(S\) such that \(\chi(u|_{\partial X}) = u|_{\partial X}\).
We will use an analogous notation for other function spaces---in particular for Sobolev spaces in which case the restriction to the boundary is to be understood in the trace sense.

Note that the operator \(\Callias_\psi\) we study here is essentially just the operator \(\ReducedSpinDirac + \iu \psi\) together with its formal adjoint \(\ReducedSpinDirac - \iu \psi\), both subject to chiral boundary conditions on the interior boundary.
The reason for considering both at the same time is simply a matter of convenience because it makes certain computations more symmetric and fits more neatly into the formal setup we considered in~\cite{Cecchini-Zeidler:ScalarMean}.

\subsection{Mass formulas}
In this subsection, we relate spectral estimates of the Callias operator \(\Callias_\psi\) to the mass.
This is a minor augmentation of the usual observation from Witten's proof of the positive mass theorem that the mass can be identified with a boundary term at infinity corresponding to Green's formula associated to the spinor Dirac operator.

Let \((\mathcal{E},g)\) be an asymptotically Euclidean end.
We say that a \(g\)-orthonormal tangent frame \((e_1, \dotsc, e_n)\) on \(\mathcal{E}\) is \emph{asymptotically constant} if there exist asymptotically Euclidean coordinates \(x = (x^1, \dotsc, x^n)\) such that \(e_i = \sum_{j} e^{j}_i \coordvf{x^j}\) satisfies \(e^{j}_i - \delta_{ij} \in \Ct^2_{-\tau}\), where \(\tau\) is the fall-off order of the end \(\mathcal{E}\).
Such an othonormal frame can always be found by orthonormalizing the coordinate frame \((\coordvf{x^1}, \dotsc, \coordvf{x^n})\) of an AE coordinate chart.
Note that any orthonormal frame on \(\mathcal{E}\) lifts to a section of the principal \(\Spin(n)\)-bundle and thus induces a trivialization of the spinor bundle \(\ReducedSpinBdl \to \mathcal{E}\).
We say that a section of the bundle \(S = \ReducedSpinBdl \oplus \ReducedSpinBdl \to \mathcal{E}\) is constant with respect to an orthonormal frame if it is constant with respect to this induced trivialization.
\begin{defi}
Let \((X,g)\) be a Riemannian spin manifold and \(\mathcal{E} \subseteq X\) an AE end.
We say that a section $u\in \SobolevH^1_{\loc}(X,S)$ is \emph{asymptotically constant in \(\mathcal{E}\)} if there exists a section \(u_0 \in \Ct^\infty(\mathcal{E},S)\) which is constant with respect to an asymptotically constant orthonormal frame such that $u|_{\mathcal{E}} - u_0 \in \SobolevH^1_{-q}(\mathcal{E},S)$, where \(q \coloneqq (n-2)/2\).
In this case, we define the norm at infinity of \(u\) in \(\mathcal{E}\) by \(|u|_{\mathcal{E}_\infty} \coloneqq |u_0| \in [0,\infty)\) (this is well-defined and independent of \(u_0\) because \(|u_0|\) is constant on \(\mathcal{E}\)).

If \((X,g)\) is a complete AE manifold, we say that $u\in \SobolevH^1_{\loc}(X,S)$ is \emph{asymptotically constant} if it is asymptotically constant in each AE end of $X$.
\end{defi}

In the next proposition, we use Callias operators to estimate the ADM-mass of AE ends.
For any smooth function $\psi$, we use the notation
\begin{equation}\label{eq:theta}
    \theta_\psi = \frac{\scal_g}{4} + \psi^2 - |\D \psi|,\qquad \eta_\psi = \frac{n-1}{2}\mean_g + \psi|_{\partial X}
\end{equation}
and
\begin{equation}\label{eq:bartheta}
    \bar{\theta}_\psi = \frac{n}{n-1}\frac{\scal_g}{4} + \psi^2 - |\D \psi|,\qquad \bar{\eta}_\psi = \frac{n}{2}\mean_g + \psi|_{\partial X}.
\end{equation}

\begin{prop}\label{prop:mass_formulas}
    Let \((X,g)\) be a complete connected asymptotically Euclidean spin manifold with compact boundary and let \(\psi \in \Cc^\infty(X,\R)\).
    Let \(u \in \SobolevH^1_{\loc}(X,S; \chi)\) be asymptotically constant.
    Then
    \begin{multline}
       \frac{n-1}{2} \omega_{n-1} \sum_{\mathcal{E}} \mass(\mathcal{E}, g) |u|_{\mathcal{E_\infty}}^2 + \|\Callias_\psi u\|_{\Lp^2(X)}^2 \geq \\
    %    = \|\nabla u\|_{\Lp^2(X)}^2 + \int_X (\frac{\scal_g}{4} + \psi^2)|u|^2 + \langle u, \clm(\D \psi) \RelDiracInv u \rangle \dV + \int_{\partial X} \eta_{\psi} |u|^2 \dS
     \geq \|\nabla u\|_{\Lp^2(X)}^2 + \int_X \theta_{\psi} |u|^2 \dV + \int_{\partial X} \eta_{\psi} |u|^2 \dS. \label{eq:mass_estimate}
    \end{multline}
    and
    \begin{multline}
       \frac{n}{2} \omega_{n-1} \sum_{\mathcal{E}} \mass(\mathcal{E}, g) |u|_{\mathcal{E_\infty}}^2 + \|\Callias_\psi u\|_{\Lp^2(X)}^2 \geq \\
     \geq \frac{n}{n-1}\|\Penrose u\|_{\Lp^2(X)}^2 + \int_X \bar{\theta}_\psi |u|^2 \dV + \int_{\partial X} \bar{\eta}_\psi |u|^2 \dS. \label{eq:mass_estimate_penrose}
    \end{multline}
    The sums on the left-hand side are taken over all the asymptotically Euclidean ends \(\mathcal{E}\).
    Moreover, \(\Penrose\) in \labelcref{eq:mass_estimate_penrose} denotes the \emph{Penrose operator} defined by \(\Penrose_\xi u = \nabla_\xi u + \frac{1}{n} \clm(\xi^\flat) \Dirac u\).
\end{prop}
\begin{proof}
Let \(\Omega_r \subseteq X\) be the compact connected domain whose boundary is the union of the interior boundary \(\partial X\) with the coordinate spheres of radius \(r > 0\) in each AE end of \(X\), where we take \(r\) to be sufficiently large for this to make sense.
Then following an analogous computation as in \cite[\S4]{Cecchini-Zeidler:ScalarMean} yields
\begin{align*}
    \int_{\Omega_r} |\Callias_\psi u|^2 \dV &= \int_{\Omega_r} |\Dirac u|^2 + \psi^2 |u|^2 + \langle u, \clm(\dd \psi) \RelDiracInv u \rangle \dV + \int_{\partial \Omega_r} \psi \langle u, \chi u \rangle \dS \\
    &= \int_{\Omega_r} |\nabla u|^2 + \frac{\scal_g}{4}|u|^2 + \psi^2 |u|^2 + \langle u, \clm(\dd \psi) \RelDiracInv u \rangle \dV \\
    &\qquad + \int_{\partial \Omega_r} \langle u, \clm(\nu^\flat) \Dirac u + \nabla_\nu u \rangle + \psi \langle u, \chi u \rangle \dS \\
    &\geq \|\nabla u\|_{\Lp^2(\Omega_r)}^2 + \int_{\Omega_r} \theta_\psi |u|^2 \dV + \int_{\partial X} \eta_\psi |u|^2 \dS \\
    &\qquad + \underbrace{\int_{\partial \Omega_r \setminus \partial X} \langle u, \clm(\nu^\flat) \Dirac u + \nabla_\nu u \rangle \dS}_{\to -\frac{n-1}{2} \omega_{n-1} \sum_{\mathcal{E}} \mass(\mathcal{E}, g) |u|_{\mathcal{E_\infty}}^2} - \underbrace{\int_{\partial \Omega_r \setminus \partial X} \psi |u|^2 \dS}_{\text{\(= 0\) for \(r \gg 1\)}},
\end{align*}
where \(\nu\) denotes the interior unit normal field.
Then letting \(r \to \infty\) proves \labelcref{eq:mass_estimate} by the usual mass computation in the spin proof of the positive mass theorem (see e.g.~\cite[Theorem~6.3]{Bartnick:MassAsymptoticallyFlat}, \cite[Corollary~5.15]{Lee2019-GeometricRelativity}) and because \(\psi\) is compactly supported.
The other estimate \labelcref{eq:mass_estimate_penrose} follows analogously using the formula \(|\nabla u|^2 = |\Penrose u|^2 + \frac{1}{n}|\Dirac u|^2\) and some rearrangement of terms, compare \cite[\S4]{Cecchini-Zeidler:ScalarMean}.
\end{proof}

We conclude this subsection with a technical lemma needed in the proof of some of our main results.

\begin{lem}\label{lem:unique_continuation}
    Let \((X,g)\) be a connected complete asymptotically Euclidean spin manifold, \(\psi\in \Cc^\infty(X,\R)\).
    Let \(u \in \SobolevH^1_\loc(X,S; \chi)\) be such that \(\Callias_\psi u = 0\) and \(\Penrose u = 0\).
    If \(u \neq 0\), then \(u(x) \neq 0\) for every \(x \in X\).
\end{lem}
\begin{proof}
    First of all, by elliptic regularity up to the boundary, see e.g.~\cite[\S7.4]{Baer-Ballmann:Boundary-value-problems-first-order}, we observe that \(u \in \Ct^\infty(X,S;\chi)\).
    Then the assumptions \(\Penrose u = 0\) and \(\Callias_\psi u = 0\) imply
    \[
        \nabla_\xi u = \Penrose_\xi u - \frac{1}{n} \clm(\xi^\flat) \Dirac u = \frac{\psi}{n} \clm(\xi^\flat) \RelDiracInv u - \frac{1}{n} \clm(\xi^\flat) \Callias_\psi u = \frac{\psi}{n} \clm(\xi^\flat) \RelDiracInv u.
    \]
    Thus \(u\) is parallel with respect to the (not necessarily metric) connection \(\widetilde{\nabla}_\xi \coloneqq \nabla_\xi - \frac{\psi}{n}\clm(\xi^\flat)\RelDiracInv\).
    Hence \(u\) satisfies a linear ordinary differential equation along each smooth path in \(X\). 
    If \(u\) vanishes at a single point it thus must vanish everywhere since \(X\) is connected.
\end{proof}

\subsection{Fredholm properties of the Callias operator on AE manifolds}
In this subsection, we discuss elliptic estimates and Fredholm properties of the Callias operator \(\Callias_\psi = \Dirac + \psi \RelDiracInv\) on AE spin manifolds.
All of this is essentially standard (see e.g.~\cite{Choquet-Bruhat_Christodoulou:EllipticSystems,Bartnick:MassAsymptoticallyFlat}, \cite[\S A.2]{Lee2019-GeometricRelativity}) but our setting is slightly different than in the available literature and so we provide quick proofs of the relevant results needed for our applications.
We deliberately do not aim for the greatest possible generality here.
In particular, we focus on compactly supported potentials, even though the results in this section would go through more generally under suitable fall-off conditions on \(\psi\).

For this entire subsection, we consider the following setup:  Let \((X,g)\) be a complete connected \(n\)-dimensional asymptotically Euclidean spin manifold with compact boundary.
Let \(S = \ReducedSpinBdl \oplus \ReducedSpinBdl\) be the relative Dirac bundle over \(X\) with associated Dirac operator \(\Dirac\) as in \cref{sec:witten_relative_dirac_bundle}. Moreover, we let \(\psi \colon X \to \R\) be a compactly supported smooth function and  consider the Callias operator \(\Callias_\psi = \Dirac + \psi \RelDiracInv\).
As before, we will mainly use the weight $-q \coloneqq -\frac{n-2}{2}$ for the Sobolev spaces in our considerations.

The first main ingredient we will use prominently is the following \emph{weighted Poincaré inequality}.
\begin{prop}[{\cite[Theorem~1.3]{Bartnick:MassAsymptoticallyFlat}, \cite[Theorem~A.28]{Lee2019-GeometricRelativity}}] \label{prop:weighted_poincare}
    % Let \((X,g)\) be a complete connected AE spin manifold with boundary and \(\delta < 0\).
    Let \(\delta < 0\) be any negative weight.
    Then there exists a constant \(C = C(X,g,\delta) > 0\) such that
    \begin{equation}
        \|u\|_{\Lp^2_{\delta}} \leq C \|\nabla u\|_{\Lp^2_{\delta-1}} \label{eq:weighted_poincare}
    \end{equation}
    for all \(u \in \SobolevH^1_{\delta}(X,S)\).
\end{prop}
To be precise, in the literature this weighted Poincaré inequality is only proved for manifolds without boundary and in the scalar case.
However, the statement for AE manifolds with compact boundary can be formally reduced to the case without boundary via a doubling argument.
Moreover, the setting of vector bundles reduces to the scalar case by Kato's inequality \(|\dd |u|| \leq |\nabla u|\).

Another standard result we need is the following weighted version of the Rellich\nobreakdash--Kondrachov compact embedding theorem, the proof of which is standard and directly extends to our setting; see~\cite[Lemma~2.1]{Choquet-Bruhat_Christodoulou:EllipticSystems}.
\begin{prop}
   The inclusion \(\SobolevH^1_{\delta}(X,S) \subset \Lp^2_{\delta'}(X,S)\) is compact if \(\delta < \delta'\).
    \label{prop:rellich}
\end{prop}

We now continue with the preparations for the main results of this subsection.
The first lemma is a suitable interior elliptic estimate needed to establish the Fredholm property.

\begin{lem} \label{lem:H1loc_estimate}
    For every compact subset \(K \subseteq X\)containing a neighborhood of \(\partial X\), there exists a constant \(C = C(X, K, g, \psi) > 0\) such that 
    \begin{equation}
        \|u\|_{\SobolevH^1_{-q}(X)} \leq C \left( \|\Callias_\psi u \|_{\Lp^2(X)} + \|u\|_{\Lp^2_{-q + \epsilon}(X)} +\|u\|_{\SobolevH^1(K)} \right) \label{eq:H1loc_estimate}
    \end{equation}
    for all \(u \in \SobolevH^1_{-q}(X,S)\), where \(\epsilon = \epsilon(n) > 0\) is some positive number depending only on the dimension \(n\).
\end{lem}
\begin{proof}
    Let \(\theta_\psi = \frac{\scal_g}{4} + \psi^2 - |\DD \psi|\).
    Since \(\psi\) is compactly supported and \(\scal_g \in \littleo(\rho^{-q-2})\) because \((X,g)\) is AE, we obtain \(\theta_\psi \in \Lp^\infty_{-2-2 \epsilon}(X)\) for \(\epsilon = q/2 = (n-2)/4\).
    Let \(K\) be any compact subset which contains a neighborhood of \(\partial X\).
    Furthermore, let \(C' = C'(X,g,-q)\) be the constant from the weighted Poincaré inequality \eqref{eq:weighted_poincare} for the weight \(\delta = -q\).
    Then for any \(u \in \SobolevH^1_{-q}(X,S)\) which vanishes on \(K\), we obtain from \labelcref{eq:mass_estimate} in \cref{prop:mass_formulas} that
    \begin{align}
        \|\Callias_\psi u\|_{\Lp^2(X)}^2 \geq& \| \nabla u \|_{\Lp^2(X)}^2 + \int_X \theta_\psi |u|^2 \dV \notag \\ 
        &\geq \frac{1}{2} \|\nabla u\|_{\Lp^2(X)}^2 + \frac{1}{2 C'^2} \|u\|_{\Lp^2_{-q}(X)}^2 - \|\theta_\psi\|_{\Lp^\infty_{-2-2\epsilon}} \|u\|_{\Lp^2_{-q+\epsilon}(X)}^2 \notag \\
        &\geq \frac{1}{C''} \|u\|_{\SobolevH^1_{-q}(X)}^2 - \|\theta_\psi\|_{\Lp^\infty_{-2-2\epsilon}} \|u\|_{\Lp^2_{-q+\epsilon}(X)}^2 \label{eq:interior_estimate}
    \end{align}
    for a suitable constant \(C'' = C''(C') > 0\).
    Finally observe that on \(K\) the \(\SobolevH^1_{-q}\)-norm can be controlled in terms the \(\SobolevH^1(K)\)-norm up to a constant depending on the weight function on \(K\).
    Thus a gluing argument using a cut-off function shows that \labelcref{eq:interior_estimate}  implies \eqref{eq:H1loc_estimate} for a suitable constant \(C\) depending on  \(C''\), \(\|\theta_\psi\|_{\Lp^\infty_{-2-2\epsilon}}\), all data on \(K\) and the derivative of the chosen cut-off function.
\end{proof}

\begin{prop}\label{prop:arbitrary_weight_estimate}
    There exists a constant \(C = C(X, g, \psi) > 0\) such that 
    \begin{equation}
        \|u\|_{\SobolevH^1_{-q}(X)} \leq C \left( \|\Callias_\psi u \|_{\Lp^2(X)} +\|u\|_{\Lp^2_{-q+\epsilon}(X)} \right) \label{eq:arbitrary_weight_estimate}
    \end{equation}
    for all \(u \in \SobolevH^1_{-q}(X,S; \chi)\), where \(\epsilon = \epsilon(n) > 0\) is some positive number depending only on the dimension \(n\).
\end{prop}
\begin{proof}
    We apply \cref{lem:H1loc_estimate} to obtain a compact subset \(K \subseteq X\) containing a neighborhood of \(\partial X\) and a constant \(C > 0\) such that \eqref{eq:H1loc_estimate} holds for \(\epsilon = \epsilon(n)\).
    Since \(\chi(u|_{\partial X}) = u|_{\partial X}\) is an elliptic boundary condition, there exists another constant \(C' > 0\) such that 
    \[
     \|u\|_{\SobolevH^1(K)} \leq C' \left(\|\Callias_\psi u\|_{\Lp^2(K)} + \|u\|_{\Lp^2(K)}\right)
    \]
    for all \(u \in \SobolevH^1_{-q}(X, S; \chi)\), see~\cite[Lemma~7.3]{Baer-Ballmann:Boundary-value-problems-first-order}.
    Since \(K\) is compact \(\|u\|_{\Lp^2(K)}\) can be controlled by \(\|u\|_{\Lp^2_{\delta}(X)}\)up to constant depending only on \(\delta = -q + \epsilon\) and the weight function on \(K\).
    Thus the desired estimate \eqref{eq:arbitrary_weight_estimate} follows from \eqref{eq:H1loc_estimate}.
\end{proof}

\begin{lem}\label{lem:kernel_of_adjoint}
    % Let \((X,g)\) be a complete connected \(n\)-dimensional AE spin manifold with compact boundary and let \(\psi \in \Cc^\infty(X,\R)\).
    Let \(u \in \Lp^2(X,S) \) be such that 
    \begin{equation}
        ( u, \Callias_\psi v )_{\Lp^2(X)} = 0 \qquad \forall v \in \Cc^\infty(X,S;\chi).
        \label{eq:adjoint_kernel_explicit}
    \end{equation}
    Then \(u \in \SobolevH^1_{-q}(X,S;\chi)\) and \(\Callias_\psi u = 0\).
\end{lem}
\begin{proof}
    We temporarily consider \(\Callias_\psi\) as an unbounded operator \(T\) on \(\Lp^2(X,S)\) with domain \(\dom(T) = \{u \in \SobolevH^1_\loc(X,S; \chi) \mid u \in \Lp^2(X,S), \Callias_\psi u \in \Lp^2(X,S)  \}\).
    Since \(\Callias_\psi\) is symmetric and \(\chi\) defines a self-adjoint elliptic boundary condition, it follows that the unbounded operator \(T\) is self-adjoint; see~\cite[\S7.2]{Baer-Ballmann:Boundary-value-problems-first-order}.
    The condition \eqref{eq:adjoint_kernel_explicit} implies that \(u \in \dom(T^\ast) = \dom(T)\) and \(0 = T^\ast u = T u\).
    In particular, \(u \in \SobolevH^1_\loc(X,S; \chi)\).
    To see the desired conclusion it thus suffices to prove that \(\|u\|_{\SobolevH^1_{-q}(X)} < \infty\).
    To this end, let \(\varphi_n \colon X \to [0,1]\) be a sequence of compactly supported smooth functions such that \(\|\dd\varphi_n\|_\infty \leq 1\) and for every compact subset \(L \subseteq X\) we have \(\varphi_n|_{L} = 1\) for all sufficiently large \(n\).
    Then, using \cref{prop:arbitrary_weight_estimate} and the fact that \(\Lp^2 = \Lp^2_{-q-1}\) continuously embeds into \(\Lp^2_{-q+\epsilon}\), we obtain a compact subset \(K \subseteq X\) and a constant \(C > 0\) such that
    \begin{align*}
    \|u\|_{\SobolevH^1_{-q}(X)} &\leq \limsup_{n \to \infty} \|\varphi_n u\|_{\SobolevH^1_{-q}(X)} \\
    &\leq \limsup_{n \to \infty}\ C \left( \|\Callias_\psi (\varphi_n u)\|_{\Lp^2(X)} +  \|\varphi_n u\|_{\Lp^2(X)} \right) \\
    &= \limsup_{n \to \infty}\ C \left( \|(\clm(\dd \varphi_n) u) \|_{\Lp^2(X)} + \|\varphi_n u\|_{\Lp^2(X)} \right) \\
    &\leq 2 C\|u\|_{\Lp^2(X)} < \infty. \qedhere
    \end{align*}
\end{proof}

Finally, the main theorem of this subsection establishes that the Callias operator \(\Callias_\psi\) is always a Fredholm operator and an isomorphism iff it has trivial kernel.

\begin{thm}\label{prop:CalliasProperties}
Let \((X,g)\) be a complete connected \(n\)-dimensional AE spin manifold with compact boundary and let \(\psi \in \Cc^\infty(X,\R)\).
Then $\Callias_\psi = \Dirac + \psi \RelDiracInv$ defines a bounded Fredholm operator
    \begin{equation}
        \Callias_\psi\colon \SobolevH^1_{-q}(X,S;\chi)\to \Lp^2(X,S) \label{eq:bounded_fredholm}
    \end{equation}
    of non-negative index, that is, \(\dim \coker(\Callias_\psi) \leq \dim \ker(\Callias_\psi)\).
    In particular, the map \labelcref{eq:bounded_fredholm} is an isomorphism if and only if it has trivial kernel.
\end{thm}
\begin{proof}
    We first note that \(\Callias_\psi = \Dirac + \psi \RelDiracInv\) defines a bounded operator \(\SobolevH^1_{-q}(X,S;\chi)\to \Lp^2(X,S)\) because \(\psi\) is compactly supported and \(\|\Dirac u\|_{\Lp^2}^2 \leq n \|\nabla u\|_{\Lp^2}^2 \leq n \|u\|_{\SobolevH^1_{-q}}^2\).
    
    Next, we prove that \(\ker(\Callias_\psi)\) is finite-dimensional.
    Indeed, let \(u_n \in \ker(\Callias_\psi)\) be any sequence such that \(\|u_n\|_{\SobolevH_{-q}^1} \leq 1\) for all \(n \in \N\).
    Then it follows from \cref{prop:rellich} that \(u_n\) admits a subsequence which converges in \(\Lp^2_{-q+\epsilon}(X,S)\) for \(\epsilon = \epsilon(n)\) as in \cref{prop:arbitrary_weight_estimate}.
    Subsequently, \cref{prop:arbitrary_weight_estimate} shows that this subsequence is Cauchy in \(\SobolevH^1_{-q}(X,S)\) and hence already converges in \(\SobolevH^1_{-q}(X,S)\).
    We thus have observed that the unit ball in \(\ker(\Callias_\psi)\) is compact, hence \(\ker(\Callias_\psi)\) must be finite-dimensional. 

    Furthermore, we claim that \(\Callias_\psi\colon \SobolevH^1_{-q}(X,S;\chi)\to \Lp^2(X,S)\) has closed range.
    Let \(V \subseteq \SobolevH^1_{-q}(X,S;\chi)\) be a closed subspace complementary to \(\ker(\Callias_\psi)\), that is, \(\SobolevH^1_{-q}(X,S;\chi) = \ker(\Callias_\psi) \oplus V\).
    To show that \(\Callias_\psi\) has closed image, it suffices to prove that there exists a constant \(c > 0\) such that \(\|\Callias_\psi v\|_{\Lp^2} \geq c \|v\|_{\SobolevH^1_{-q}}\) for all \(v \in V\).
    If such a constant did not exist, then there would be a sequence \(v_n \in V\) such that \(\|v_n\|_{\SobolevH^1_{-q}} = 1\) and \(\|\Callias_\psi v_n\|_{\Lp^2} \to 0\).
    Again, \(v_n\) admits a subsequence converging to an element \(v \in \Lp^2_{-q+\epsilon}(X)\) and \cref{prop:arbitrary_weight_estimate} implies that this convergence actually takes place in \(\SobolevH^1_{-q}\).
    In particular, \(v \in V\) as \(V\) is closed.
    By continuity, we have \(\|v\|_{\SobolevH^1_{-q}} =1\) and \(\Callias_\psi v = 0\), but this is a contradiction since \(V\) is complementary to the kernel.

    Thus we have seen that \(\Callias_\psi\) has finite-dimensional kernel and closed range.
    Finally, the adjoint (or to be precise, the \emph{dual operator}) of \(\Callias_\psi\) can be viewed as an operator
    \[
        \Callias_\psi^\ast \colon \Lp^2(X,S) \to \left(\SobolevH^1_{-q}(X,S;\chi)\right)^\ast,
    \]
    where we implicitly identify \(\Lp^2(X,S)\) with its dual space via the \(\Lp^2\)-pairing.
    \Cref{lem:kernel_of_adjoint} implies that \(\ker(\Callias_\psi^\ast) \subseteq \ker(\Callias_\psi) \subseteq \SobolevH^1_{-q}(X,S;\chi)\).
    Thus \(\dim \coker(\Callias_\psi) = \dim \ker(\Callias_\psi^\ast) \leq \dim \ker(\Callias_\psi) < \infty\) and the theorem is proved.
\end{proof}

\section{The neck of AE manifolds of nonpositive mass}\label{sec:Long_Neck}
This section is devoted to proving \cref{thm:C}.
Let $(X,g)$ be an $n$-dimensional asymptotically Euclidean manifold with compact boundary.
For a potential $\psi$, let $\bar\theta_\psi$ and $\bar\eta_\psi$ be the functions defined by~\eqref{eq:bartheta}.
Our proof of \cref{thm:C} is based on the following abstract criterion for positivity of the mass.

\begin{thm}\label{thm:AbstractPositivity}
Let $(X,g)$ be a complete asymptotically Euclidean spin manifold with compact boundary and let $\psi \in \Ct^\infty_\cpt(X,\R)$.
Suppose that $\bar\theta_\psi\geq 0$ and $\bar\eta_\psi\geq 0$.
Then the ADM-mass of each asymptotically Euclidean end is non-negative.
If, in addition, $\bar\theta_\psi(x)>0$ or $\bar \eta_\psi(x) > 0$ for some point \(x \in X\) or \(x \in \partial X\), respectively, then the ADM mass of each asymptotically Euclidean end is strictly positive.
\end{thm}

\noindent Before proving \cref{thm:AbstractPositivity}, let us show that it implies \cref{thm:C}.
The next two lemmas will be used to construct suitable potentials.

\begin{lem}\label{lem:tangent_growth}
    Let \(V\) be a compact manifold with boundary such that \(\partial V = \partial_- V \sqcup \partial_+ V\), where \(\partial_\pm V\) are non-empty unions of components.
    Let \(\eta\), \(\delta\) be positive constants such that $\eta\delta<\frac{\pi}{2}$ and suppose that \(\dist_g(\partial_- V, \partial_+ V) > \delta\).
    Then there exists a smooth function \(p \colon V \to \left[0,\eta\tan\left(\eta\delta\right) \right]\) such that \(p = 0\) in a neighborhood of \(\partial_- V\), \(p = \eta\tan\left(\eta\delta\right)\) in a neighborhood of \(\partial_+ V\) and
    \( \eta^2+p^2-|\D p|\geq 0.\)
\end{lem}

\begin{proof}
    As \(\dist_g(\partial_- V, \partial_+ V) > \delta\), there exists a smooth \(1\)-Lipschitz function \(x \colon V \to \R\) such that \(x|_{\partial_- V} = -\varepsilon\) and \(x|_{\partial_+ V} = \delta + \varepsilon\), for some \(\varepsilon > 0\), see e.g.~\cite[Lemma~7.2]{Cecchini-Zeidler:ScalarMean}.
    Consider the smooth function \(\tilde p(t)\coloneqq\eta\tan\left(\eta t\right)\), for $t\in[0,\delta]$.
    Note that $\tilde p(0)= 0$, $\tilde p(\delta)=\eta\tan\left(\eta\delta\right)$, and $\eta^2+\tilde p^2(t)-\tilde p^\prime(t)=0$.
    By slightly modifying $\tilde p$, we obtain a smooth function $p_\varepsilon\colon \left[-\varepsilon,\delta+\varepsilon]\to[0,\eta\tan\left(\eta\delta\right)\right]$ such that $p_\varepsilon=0$ in a neighborhood of~$-\varepsilon$, $p_\varepsilon=\eta\tan\left(\eta\delta\right)$ in a neighborhood of $\delta+\varepsilon$, and $\eta^2+p_\varepsilon^2-p_\varepsilon^\prime\geq 0$.
    Finally, by setting $p\coloneqq p_\varepsilon\circ x$, we obtain a smooth function on $V$ with the desired properties.
\end{proof}

\begin{lem}\label{lem:escape_to_infinity}
    Let \(V\) be a compact manifold with boundary such that \(\partial V = \partial_- V \sqcup \partial_+ V\), where \(\partial_\pm V\) are non-empty unions of components.
    Let \(\lambda > 0\) be arbitrary and suppose that \(\dist_g(\partial_- V, \partial_+ V) >\delta\), for some $\delta\in\left(0,\frac{1}{\lambda}\right)$.
    Then there exists a smooth function \(h \colon V \to [\lambda, \infty)\) such that \(h = \lambda\) in a neighborhood of \(\partial_- V\), \(h = \frac{\lambda}{1-\delta\lambda}\) in a neighborhood of \(\partial_+ V\), and \( h^2-|\D h|\geq 0\).
\end{lem}

\begin{proof}
    As \(\dist_g(\partial_- V, \partial_+ V) > \delta\), in the same fashion as in the proof of \cref{lem:tangent_growth}, pick a smooth \(1\)-Lipschitz function \(x \colon V \to \R\) such that \(x|_{\partial_- V} = -\varepsilon\) and \(x|_{\partial_+ V} = \delta + \varepsilon\), for some \(\varepsilon > 0\).
    Consider the smooth function $\tilde h(t)\coloneqq\frac{\lambda}{1-\lambda t}$, for $t\in\left[0,{1}/{\lambda}\right)$.
    Note that $\tilde h(0)= \lambda$, \(\tilde h(\delta)=\frac{\lambda}{1-\lambda\delta}\), and \((\tilde{h} \circ x)^2 - |\dd (\tilde{h} \circ x)| = \frac{\lambda^2}{(1-\lambda x)^2}(1- |\dd x|) \geq 0\).
    By slightly modifying $\tilde h$ in the same fashion as in the proof of \cref{lem:tangent_growth} and composing with $x$, we obtain a smooth function $h$ on $V$ with the desired properties.
\end{proof}

\begin{proof}[Proof of \cref{thm:C}]
By continuity, we can find \(d' < d\) and \(l' < l\) such that \(d' < \frac{\pi}{\sqrt{\kappa} n}\), \(l' < \frac{1}{\lambda(d')}\), and \(\mean_g > -\Psi(d', l') > -\infty\).
Since \(\dist_g(\partial X_0, \partial X_1)>d'\), by \cref{lem:tangent_growth} there exists a smooth function \(p \colon X_1\setminus X_0\to \left[0,\lambda(d')\right]\) such that \(p = 0\) in a neighborhood of \(\partial X_0\), \(p = \lambda(d')\) in a neighborhood of \(\partial X_1\) and
\begin{equation}\label{eq:MeanScalar2}
    \frac{\kappa n^2}{4}+p^2-|\dd p|\geq0.
\end{equation}
Since \(\dist_g(\partial X_1, \partial X)>l'\), by \cref{lem:escape_to_infinity} there exists a smooth function $h\colon X\setminus X_1\to[\lambda(d'),\infty)$ such that \(h = \lambda(d')\) in a neighborhood of \(\partial X_1\), \(h|_{\partial X}\geq \frac{\lambda(d')}{1-l'\lambda(d')}=\frac{n}{2}\Psi(d',l')\) and
\begin{equation}\label{eq:MeanScalar3}
    h^2-|\D h|\geq 0.
\end{equation}
Finally, let $\psi\in\Ct^\infty_\cc(X,\R)$ be defined by setting $\psi|_{X_0}=0$, $\psi|_{X_1\setminus X_0}=p$, and $\psi|_{X\setminus X_1}=h$. 
Since $\psi=0$ in $X_0$ and using~\eqref{eq:MeanScalar2} and \eqref{eq:MeanScalar3}, we have $\bar\theta_\psi\geq 0$.
Since \(h|_{\partial X}\geq \frac{n}{2}\Psi(d',l')\), $\bar\eta_\psi>0$.
Therefore, \cref{thm:AbstractPositivity} yields the conclusion.
\end{proof}

\noindent Let us now prove \cref{thm:AbstractPositivity}.
The next lemma is needed to estimate the mass.

\begin{lem}\label{lem:ScalPositiveIsomorphism}
Let $(X,g)$ be an asymptotically Euclidean spin manifold with compact boundary and let $\psi \in \Ct^\infty_\cpt(X,\R)$.
Suppose that $\bar\theta_\psi\geq 0$ and $\bar\eta_\psi\geq 0$.
Then the operator $\Callias_\psi\colon \SobolevH^1_{-q}(X,S;\chi)\to \Lp^2(X,S)$ is an isomorphism.
\end{lem}

\begin{proof}
Let $u\in \SobolevH^1_{-q}(X,S;\chi)$ be a section in the kernel of $\Callias_\psi$.
Since $\theta_\psi\geq 0$ and $\eta_\psi\geq 0$, using \labelcref{eq:mass_estimate_penrose} yields 
\[
    0=\|\Callias_\psi(u)\|^2_{\Lp^2(X)} \geq \|\Penrose u\|^2_{\Lp^2(X)} + \int_X\theta_\psi|u|^2\dV+\int_{\partial X}\eta_\psi|u|_{\partial X}|^2\dS\geq 0.
\]
Thus \(\Penrose u = 0\), or in other words, \(\nabla_\xi u = -\frac{1}{n} \clm(\xi^\flat) \Dirac u = \frac{\psi}{n} \clm(\xi^\flat) \RelDiracInv u\) for every vector field \(\xi\).
Since \(\psi\) is compactly supported, this implies that \( \nabla u = 0\) near infinity in each end.
But since \(u\) decays at infinity, this implies that \(u = 0\) near infinity in each end.
But then \cref{lem:unique_continuation} implies that \(u\) vanishes everywhere.
Hence the kernel of $\Callias_\psi$ is trivial and \cref{prop:CalliasProperties} implies the conclusion.
\end{proof}

\noindent We are now ready to prove \cref{thm:AbstractPositivity}.

\begin{proof}[Proof of \cref{thm:AbstractPositivity}]
Fix any asymptotically Euclidean end $\mathcal E$ of $X$ and let $\xi_0\in \Ct^\infty(X,S)$ be an asymptotically constant section which is non-zero at infinity in \(\mathcal{E}\) but vanishes on all other ends.
Note that $\Callias_\psi(\xi_0)$ is a smooth section in $\Lp^2(X,S)$.
By \cref{lem:ScalPositiveIsomorphism}, there exists a section $\xi\in \SobolevH^1_{-q}(X,S)$ such that $\Callias_\psi(\xi)=\Callias_\psi(\xi_0)$.
Define $u\coloneqq\xi-\xi_0$.
Then \(\Callias_\psi u = 0\) and hence \(u\) is smooth by elliptic regularity.
Moreover, using \labelcref{eq:mass_estimate_penrose} from \cref{prop:mass_formulas}, shows that
\begin{equation}\label{eq:abstract_positivity_estimate}
    \frac{n}{2} \omega_{n-1} \mass(\mathcal{E},g) |u|_{\mathcal{E}_\infty}^2
    \geq  \frac{n}{n-1}\|\Penrose u\|_{\Lp^2(X)}^2  + \int_X \bar\theta_\psi |u|^2\dV+\int_{\partial X}\bar\eta_\psi|u|^2\dS 
\end{equation}
Thus \(\mass(\mathcal{E},g) \geq 0\) because \(\bar{\theta}_\psi \geq 0\) and \(\bar{\eta}_\psi \geq 0\).
Finally, if \(\mass(\mathcal{E},g) = 0\), then \labelcref{eq:abstract_positivity_estimate} implies furthermore that \(\Penrose u = 0\).
Since $\xi_0$ is non-zero at infinity in \(\mathcal{E}\), we have $u\neq 0$ and so \cref{lem:unique_continuation} shows that $u(x)\neq 0$ for every $x\in X$.
But then \labelcref{eq:abstract_positivity_estimate} implies that \(\bar{\theta}_\psi = 0\) and \(\bar{\eta}_\psi = 0\) everywhere.
\end{proof}
\section{The positive mass theorem with arbitrary ends}\label{sec:PMTAE}
This section is devoted to proving \cref{thm:A,thm:B}.
Let $(X,g)$ be an asymptotically Euclidean manifold with compact boundary.
For \(\psi \in \Cc^\infty(X,\R)\), let $\theta_\psi$ and $\eta_\psi$ be the functions defined by~\eqref{eq:theta}.
In the next theorem, we estimate sections in $\SobolevH^1_{-q}(X,S;\chi)$ in the case that $\theta_\psi$ is possibly negative in a suitable region.

\begin{thm}\label{thm:isomorphism}
    Let \((X,g)\) be a complete connected asymptotically Euclidean spin manifold with compact boundary and let \(\psi \in \Cc^\infty(X,\R)\) be such that $\eta_\psi\geq 0$.
    Write \(\theta_\psi = \theta_+ - \theta_-\) with \(\theta_\pm \geq 0\).
    Suppose that \(\supp(\theta_-) \subseteq X_0\) for a connected codimension zero submanifold \(X_0 \subseteq X\) with compact boundary which contains at least one asymptotically Euclidean end of \(X\).
    Then there exists a constant \(c_0 = c_0(X_0,g) > 0\), depending only on \((X_0, g)\), such that if 
    \[
    \| \theta_- \|_{\Lp^\infty_{-2}(X_0)} \leq c_0,
    \]
    then 
    \[
        \Callias_\psi\colon \SobolevH^1_{-q}(X,S; \chi)\to \Lp^2(X,S)
    \]
    is an isomorphism and
    \begin{equation}
        \|\Callias_\psi u\|_{\Lp^2(X)}^2 \geq c_0 \|u\|_{\Lp^2_{-q}(X_0)}^2\label{eq:partial_injectivity_estimate}
    \end{equation}
    for all \(u \in \SobolevH^1_{-q}(X;S;\chi)\).
\end{thm}

\begin{proof}
    We use \cref{prop:weighted_poincare} to find a constant \(C = C(X_0, g, -q)> 0\) such that \(\| v\|_{\Lp^2_{-q}(X_0)}^2 \leq C \|\nabla v\|_{\Lp^2(X_0)}^2\) for all \(v \in \SobolevH^1_{-q}(X_0,S)\).
    We then set \(c_0 \coloneqq 1/(2C) > 0\).
    Then, using \labelcref{eq:mass_estimate}, we arrive at the following estimate for all \(u \in \SobolevH^1_{-q}(X,S;\chi)\).
    \begin{align}
        \|\Callias_\psi u\|_{\Lp^2(X)}^2 &\geq \int_X |\nabla u|^2 + \theta_\psi |u|^2 \dvol + \int_{\partial X} \underbrace{\left(\frac{n-1}{2}\mean_g + \psi\right)}_{= \eta_\psi \geq 0}|u|^2 \dS \notag \\ 
        &\geq \|\nabla u\|^2_{\Lp^2(X)} - \|\rho^2 \theta_-\|_{\Lp^\infty} \cdot \int_{X_0} \rho^{-2} |u|^2 \dV \notag \\
        &=\|\nabla u\|^2_{\Lp^2(X)} - \|\theta_-\|_{\Lp^\infty_{-2}} \cdot \|u\|_{\Lp^2_{-q}(X_0)}^2 \label{eq:converse_connection_estimate}\\
        &\geq\|\nabla u\|^2_{\Lp^2(X_0)} - \|\theta_-\|_{\Lp^\infty_{-2}} \cdot \|u\|_{\Lp^2_{-q}(X_0)}^2 \notag \\
        &\geq \left(\frac{1}{C} -  c_0\right) \|u\|_{\Lp^2_{-q}(X_0)}^2 = c_0 \|u\|_{\Lp^2_{-q}(X_0)}^2. \notag
    \end{align}
    This already proves \labelcref{eq:partial_injectivity_estimate}.
    Now to see that \(\Callias_\psi\) is an isomorphism it is enough to show that \(\Callias_\psi\) has trivial kernel by \cref{prop:CalliasProperties}.
    Indeed, for \(u \in \ker(\Callias_\psi)\), each line in the above estimate must vanish.
    In particular, \(\|u\|_{\Lp^2_{-q}(X_0)} = 0\) and vanishing of \labelcref{eq:converse_connection_estimate} subsequently shows that \(\|\nabla u\|_{\Lp^2(X)}^2 = 0\). 
    Hence \(u\) is parallel on all of \(X\). Since \(u\) vanishes on \(X_0\) this proves that \(u = 0\) everywhere.
    Thus \(\ker(\Callias_\psi) = 0\).
\end{proof}

\noindent We are now ready to prove \cref{thm:A,thm:B}.

\begin{proof}[Proof of \cref{thm:A}]
Let $X_0\subseteq \mathcal{E}$ be the complement of some open Euclidean ball in \(\mathcal{E}\).
We furthermore assume that some neighborhood of \(X_0\) in \(M\) is still complete, spin and has non-negative scalar curvature (otherwise there is nothing to prove).
Let $K \subset X_0$ be the closure of a collar neighborhood of $\partial X_0$ inside \(X_0\).
Furthermore, fix a smooth function $f\colon X_0\to [0,1]$ such that $f=0$ in $X_0\setminus K$ and $f=1$ in a small neighborhood of \(\partial X_0\).
For later use we furthermore fix an arbitrary section \(\xi_0 \in \Ct^\infty(X_0,S)\) which is asymptotically constant in \(\mathcal{E}\) with \(|\xi_0|_{\mathcal{E}_\infty} \neq 0\) and such that \(\supp(\xi_0) \subseteq X_0 \setminus K\).

Next we let \(\lambda_0 \in [0,\infty)\) be the infimum of all positive real numbers \(\lambda > 0\) such that there exists a complete codimension zero spin submanifold $X_\lambda\subset M$ with compact (possibly empty) boundary such that $X_0\subset X_\lambda$, $X_\lambda\setminus X_0$ is relatively compact in $M$, \(\scal_g \geq 0\) on \(X_\lambda\), and $\dist_g(\partial X,\partial X_\lambda)>\frac{1}{\lambda}$ (we take this condition to be trivially satisfied if \(\partial X_{\lambda} = \emptyset\)).
For later use, we fix such a submanifold \(X_\lambda\) for each \(\lambda > \lambda_0\).
Heuristically, the constant \(\lambda_0\) is the smallest non-negative number such that all the conditions \labelcref{item:complete,item:geq_scal,item:spin} from the statement of \cref{thm:A} still hold in a \(\frac{1}{\lambda_0}\)-neighborhood of \(X_0\) in \(M\).
To prove \cref{thm:A}, we must thus demonstrate that \(\lambda_0\) admits a positive lower bound which only depends on the end \((\mathcal{E},g)\).

Using \cref{lem:escape_to_infinity}, for each \(\lambda > \lambda_0\), we can choose a smooth function $h_\lambda\colon X_\lambda\setminus X_0\to[\lambda,\infty)$ such that $h_\lambda=\lambda$ in a small neighborhood of $\partial X_0$ which satisfies
\[ h_\lambda^2-|\D h_\lambda|\geq 0 \] 
in all of $X_\lambda\setminus X_0$, and $h_\lambda|_{\partial X_\lambda} + \frac{(n-1) H_\lambda}{2}\geq 0$, where $H_\lambda$ is the mean curvature of $\partial X_\lambda$ with respect to the inwards pointing unit normal field.
Let $\psi_\lambda\colon X_\lambda\to[0,\infty)$ be the smooth function defined by
\[
    \psi_\lambda(x)\coloneqq\begin{cases}
    \lambda f(x)&x\in X_0\\h_\lambda(x)& x\in X_\lambda\setminus X_0.
    \end{cases}
\]
Let $\Callias_{\psi_\lambda}$ be the associated Callias operator and let $\theta_\lambda = \frac{\scal_g}{4} + \psi^2_\lambda - |\D \psi_{\lambda}|$, $\eta_\lambda = h_\lambda|_{\partial X_\lambda} + \frac{(n-1) H_\lambda}{2}$.
By construction, we thus obtain
\begin{myenumi}
    \item $\theta_\lambda\geq 0$ in $X_\lambda\setminus K$;\label{item:LPMT1}
    \item $\theta_\lambda\geq-\lambda |\D f|$ in $K$;\label{item:LPMT2}
   % \item $\theta_\lambda\geq 0$ in $X_\lambda\setminus X_0$;\label{item:LPMT3}
    \item $\eta_\lambda\geq 0$.\label{item:LPMT4}
\end{myenumi}
In particular, the negative part of \(\theta_\lambda\) is supported in \(K \subseteq X_0\) and satisfies \((\theta_\lambda)_- \leq \lambda |\dd f|\) where it is non-zero.
Note also that, by \cref{thm:isomorphism}, there exist constants $\lambda_1 = \lambda_1(X_0,g)>0$ and \(c_0 = c_0(X_0,g)> 0\) depending only on the data in \(X_0\) such that
\begin{equation}
    \forall \lambda \in (\lambda_0, \lambda_1) \quad \forall v \in \SobolevH^1_{-q}(X_\lambda, S; \chi): \qquad \|\Callias_{\psi_{\lambda}} v\|_{\Lp^2(X_\lambda)}^2 \geq c_0 \|v\|_{\Lp^2_{-q}(X_0)}^2, \label{eq:LPMT_partial_injectivity_estimate}
\end{equation}
and the operator $\Callias_{\psi_\lambda}\colon \SobolevH^1_{-q}(X_\lambda,S;\chi)\to \Lp^2(X_\lambda,S)$ is an isomorphism for all $\lambda\in(\lambda_0,\lambda_1)$.
Note that we can assume without loss of generality that \(\lambda_0 < \lambda_1\) because if \(\lambda_0 \geq \lambda_1\), then the desired conclusion of the theorem would already be proved for the constant \(R = \frac{1}{\lambda_1}\).

Next, let $\lambda\in(\lambda_0,\lambda_1)$ be arbitrary.
We now turn to the asymptotically constant section \(\xi_0\) which has been fixed in advance.
Since by construction \(\supp(\xi_0) \cap \supp(\psi_\lambda) = \emptyset\), we have $\Callias_{\psi_\lambda}(\xi_0)=\Dirac(\xi_0)$.
Moreover, \(\Dirac(\xi_0) \in \Lp^2(X_0, S) \subset \Lp^2(X_\lambda, S)\).
Thus there exist $\xi_\lambda\in \SobolevH^1_{-q}(X_\lambda,S; \chi)$ such that 
\begin{equation}\label{eq:LPMT1}
    \Callias_{\psi_\lambda}(\xi_\lambda)= -\Dirac(\xi_0) = -\Callias_{\psi_{\lambda}}(\xi_0) .
\end{equation}
Define $u_\lambda\coloneqq\xi_0+\xi_\lambda$.
Then $\Callias_{\psi_\lambda}(u_\lambda)=0$ everywhere and \(u_\lambda = \xi_{\lambda}\) on \((X_{\lambda} \setminus X_0) \cup K\).
By \labelcref{eq:mass_estimate} and using~\labelcref{item:LPMT1,item:LPMT2,item:LPMT4}, we deduce
\begin{align}
    \frac{n-1}{2} \omega_{n-1}\cdot\mass(\mathcal E,g) |\xi_0|_{\mathcal{E}_\infty}^2 &\geq \|\nabla u_\lambda\|_{\Lp^2(X_\lambda)} + \int_{X_\lambda}\theta_\lambda  |u_\lambda|^2\dV
    +\int_{\partial X_{\lambda}}\eta_\lambda|u_\lambda|^2\dS \notag \\
     &\geq \|\nabla u_\lambda\|_{\Lp^2(X_\lambda)} -\int_{X_{0}}(\theta_\lambda)_{-}|u_\lambda|^2\dV \notag \\
    &\geq  \|\nabla u_\lambda\|_{\Lp^2(X_\lambda)} -\lambda \|\D f\|_{\Lp^\infty_{-2}}\|\xi_\lambda\|_{\Lp^2_{-q}(X_0)}^2 \label{eq:LPMT2}
\end{align}
where in the last inequality we used that $u_\lambda=\xi_\lambda$ on $K$.
By \labelcref{eq:LPMT_partial_injectivity_estimate}, we obtain
\begin{equation}\label{eq:LPMT4}
    \|\Dirac(\xi_0)\|_{\Lp^2(X_0)}^2 = \|\Dirac(\xi_0)\|_{\Lp^2(X_\lambda)}^2= \|\Callias_{\psi_\lambda}(\xi_\lambda)\|_{\Lp^2(X_\lambda)}^2 \geq c_0 \|\xi_\lambda\|_{\Lp^2_{-q}(X_0)}^2.
    \end{equation}
Combining \labelcref{eq:LPMT2,eq:LPMT4}, we deduce that
\begin{equation}
    \frac{n-1}{2} \omega_{n-1} \cdot \mass(\mathcal E,g) |\xi_0|_{\mathcal{E}_\infty}^2\geq \|\nabla u_\lambda\|_{\Lp^2(X_\lambda)}^2 - \lambda \|\D f\|_{\Lp^\infty_{-2}(X_0)} \frac{\|\Dirac\xi_0\|_{\Lp^2(X_0)}^2}{{c_0}}.
    \label{eq:arbitrary_ends_mass_estimate}
\end{equation}
for all \(\lambda\in(\lambda_0, \lambda_1)\).
Then letting $\lambda \searrow \lambda_0$ proves that
\[
    \lambda_0 \cdot \|\dd f\|_{\Lp^\infty_{-2}(X_0)} \frac{\|\Dirac\xi_0\|_{\Lp^2(X_0)}^2}{{c_0}} \geq -\frac{n-1}{2} \omega_{n-1} \cdot \mass(\mathcal E,g) |\xi_0|_{\mathcal{E}_\infty}^2.
\]
Note that both \(\xi_0\) and \(f\) live only on \(X_0 \subset \mathcal{E}\) and have been chosen in advance independently of \(\lambda_0\).
Moreover, if \(\Dirac(\xi_0)\) was zero, then \(\xi_\lambda = 0\) for each \(\lambda \in (\lambda_0, \lambda_1)\) and \labelcref{eq:LPMT2} would imply that \(\mass(\mathcal{E}, g) \geq 0\).
Thus, if \(\mass(\mathcal{E},g) < 0\) as in the hypothesis of \cref{thm:A}, this proves that \(\lambda_0\) admits a strictly positive lower bound which only depends on the end \((\mathcal{E},g)\) and thereby finishes the proof of \cref{thm:A}.
\end{proof}

\begin{proof}[Proof of \cref{thm:B}]
Applying the same construction and notation as in the proof of \cref{thm:A}, we must have \(\lambda_0 = 0\) since \((M,g)\) is complete, spin and has non-negative scalar curvature everywhere.
Then letting \(\lambda \searrow 0 = \lambda_0\) in \labelcref{eq:arbitrary_ends_mass_estimate} proves that \(\mass(\mathcal{E},g) \geq 0\) as desired (alternatively, this is also already a formal consequence of the statement of \cref{thm:A}).

Finally, we turn to the rigidity statement and assume that \(\mass(\mathcal E,g) = 0\).
Then fix an arbitrary codimension zero submanifold \(X \subseteq M\) with boundary such that \(X_0 \subseteq X\) and \(X \setminus X_0\) is relatively compact.
Then for all \(\lambda > 0\) sufficiently close to zero, we have \(X \subset X_\lambda\).
Thus we obtain from \labelcref{eq:arbitrary_ends_mass_estimate} that \(\|\nabla u_\lambda\|_{\Lp^2(X)}^2 \leq \|\nabla u_\lambda\|^2_{\Lp^2(X_\lambda)} \to 0\) as \(\lambda \searrow 0\).
Since \(\xi_\lambda \in \SobolevH^{1}_{-q}\), the weighted Poincaré inequality on \(X\) (see \cref{prop:weighted_poincare}) implies that there exists a constant \(C_X\), which only depends on all the data on \(X\), such that
\[
\|\xi_{\lambda} - \xi_{\lambda'}\|_{\SobolevH^1_{-q}(X)} \leq C_X \|\nabla \xi_{\lambda} - \nabla \xi_{\lambda'}\|_{\Lp^2(X)} = C_X\|\nabla u_{\lambda} - \nabla u_{\lambda'}\|_{\Lp^2(X)} \to 0
\]
as \(\lambda, \lambda' \searrow 0\).
Thus \(\xi \coloneqq \lim_{\lambda \searrow 0} \xi_{\lambda}\) exists in \(\SobolevH^1_{-q}(X,S)\).
Setting \(u \coloneqq \xi_0 + \xi\), we thus obtain a section \(u \in \SobolevH^1_{\loc}(X,S)\) such that \(\nabla u = 0\) and \(u - \xi_0 \in \SobolevH^1_{-q}(X, S)\).
Since \(X\) was arbitrary, this actually shows that we obtain a parallel section \(u\) of \(S\) defined on all of \(M\) which is asymptotic to \(\xi_0\) in the end \(\mathcal{E}\).
This implies that we can find parallel spinors on \(M\) which are asymptotic to arbitrarily chosen constant spinors in the end \(\mathcal{E}\) because in the construction above the asymptotically constant section \(\xi_0\) was arbitrary and \(S = \ReducedSpinBdl \oplus \ReducedSpinBdl\) by definition.
As in the classical spin proof of the rigidity part of the positive mass theorem, this implies the existence of a globally parallel tangent frame on \(M\) (see for instance~\cite[692]{Bartnick:MassAsymptoticallyFlat}, \cite[175]{Lee2019-GeometricRelativity}) and so \((M,g)\) is flat. Then an elementary argument using the Cartan--Hadamard theorem implies that \((M,g)\) is isometric to Euclidean space.
\end{proof}

\printbibliography
\end{document}